\documentclass[12pt,a4paper]{amsart}
\usepackage[T1]{fontenc} 
\usepackage[utf8]{inputenc} 
\usepackage{xy} 
\usepackage[dvips]{graphicx}
\newtheorem{thm}{Theorem}[section]

\newtheorem{lem}[thm]{Lemma}
\newtheorem{prop}[thm]{Proposition}
\theoremstyle{definition}
\newtheorem{defn}[thm]{Definition}
\theoremstyle{remark}
\newtheorem{expl}[thm]{Example}
\newtheorem{rem}[thm]{Remark}

\numberwithin{equation}{section}

\newcommand{\brqt}[1]{\left\langle#1\right\rangle}
\newcommand{\prth}[1]{\left(#1\right)}

\newcommand{\set}[1]{\left\{#1\right\}}
\newcommand{\RR}{\mathbb R}

\newcommand{\Rn}{\mathbb R^n}

\newcommand{\Mapsto}{\longmapsto}
\newcommand{\To}{\longrightarrow}
\newcommand{\Iso}{\mathrm{Iso}}
\newcommand{\Gl}{\mathrm{Gl}}
\newcommand{\Sl}{\mathrm{Sl}}
\newcommand{\GG}{\mathcal G} 
\newcommand{\GB}{\mathcal G(B)} 
\newcommand{\OO}{\mathcal O} 
\newcommand{\UU}{\mathcal U} 
\newcommand{\NN}{\mathcal N} 
\newcommand{\FF}{\mathcal F} 
\newcommand{\FM}{\mathcal F M} 
\newcommand{\FB}{\mathcal F B} 

\renewcommand{\and}{\quad\textrm{and}\quad}
\newcommand{\cf}{\emph{cf.}}
\newcommand{\eg}{\emph{e.g.}}

\newcommand{\ie}{\emph{i.e.}}

\newcommand{\secref}[1]{\S\ref{#1}}
\newcommand{\itref}[1]{(\ref{#1})}

\begin{document}
\title{Functionally Graded Media}
\author[C. M. Campos]{C\'edric M. Campos}
\address{Instituto de Ciencias Matem\'aticas\\CSIC-UAM-UC3M-UCM\\Serrano 123\\28006 Madrid\\Madrid
(Spain)} \email{cedricmc@imaff.cfmac.csic.es}
\author[M. Epstein]{Marcelo Epstein}
\address{Department of Mechanical and Manufacturing Engineering\\The
University of Calgary, Calgary, Alberta, Canada T2N 1N4}
\email{mepstein@ucalgary.ca}
\author[M. de Le\'on]{Manuel de Le\'on}
\address{Instituto de Ciencias Matem\'aticas\\CSIC-UAM-UC3M-UCM\\Serrano 123\\28006 Madrid\\Madrid
(Spain)} \email{mdeleon@imaff.cfmac.csic.es}
\begin{abstract}
The notions of uniformity and homogeneity of elastic materials are
reviewed in terms of Lie groupoids and frame bundles. This
framework is also extended to consider the case Functionally
Graded Media, which allows us to obtain some homogeneity
conditions.
\end{abstract}
\maketitle

\section{Introduction} \label{sec.intro}
The mechanical response at a point $X$ of a simple (first-grade)
local elastic body $B$ depends on the first derivative $F$ at
$X\in B$ of the deformation. In other words, $B$ obeys a
constitutive law of the form:
\begin{equation} \label{claw}
W = W(F(X); X)
\end{equation}
where $W$ measures the strain energy per unit volume. The linear
map $F(X)$ is called the deformation gradient at $X$. Of course,
there are materials for which the constitutive equation implies
higher order derivatives or even internal variables as it happens
with the so-called Cosserat media or, more generally, media with
microstructure, but such materials will not be considered here.

An important problem in Continuum Mechanics is to decide if the
body is made of the same material at all its points. To handle
this question in a proper mathematical way, one introduces the
concept of material isomorphism, that is, a linear isomorphism
$P_{XY} : T_XB \To T_YB$ such that
$$ W(FP_{XY}; X) = W(F; Y) $$
for all deformation gradients $F$ at $Y$. Intuitively, this means
that we can extract a small piece of material around $X$ and
implant it into $Y$ without any change in the mechanical response
at $Y$. If such is the case for all pairs of body points, we say
that the body $B$ is uniform. This has been the starting point of
the work by Noll and Wang \cite{Noll,TrueNoll65,TrueWang73,Wang67}
in their approach to uniformity and homogeneity.

In this context, a material symmetry at $X$ is nothing but a
material automorphism of the tangent space $T_XB$. The collection
of all the material symmetries at $X$ forms a group, the material
symmetry group $\GG(X)$ at $X$. An important consequence of the
uniformity property is that the material symmetry groups at two
different points $X$ and $Y$ are conjugate.

A natural question arises: Is there a more general notion that
permits to compare the material responses at two arbitrary points
even if the body does not enjoy uniformity? An answer to this
question is based on the comparison of the symmetry groups at
different points. Indeed, we say that the body $B$ is unisymmetric
if the material symmetry groups at two different points are
conjugate, whether or not the points are materially isomorphic.
From the point of view of applications, this kind of body
corresponds to certain types of the so-called functionally graded
materials (FGM for short). The unisymmetry property was introduced
in \cite{EpsLe00} with the objective to extend the notion of
homogeneity to non-uniform material bodies. Let us recall that the
homogeneity of a uniform body is equivalent to the integrability
of the associated material $G$-structure
\cite{Bloom79,ElEpsSni90}. Roughly speaking, this material
$G$-structure is obtained by attaching to each point of $B$ the
corresponding material symmetry group via the choice of a given
linear reference at a fixed point; a change of the linear
reference gives a conjugate $G$-structure. In a more sophisticated
framework, the set of all material isomorphisms defines a Lie
groupoid, which in some sense is a way to deal with all these
conjugate $G$-structures at the same time.

In the case of unisymmetric materials the attached group is not
the material symmetry group, but its normalizer within the whole
general linear group. This implies a more difficult understanding
of the generalized concept of homogeneity associated with
unisymmetric materials. The main aim of the present paper is to
provide a convenient characterization of this homogeneity
property. In this sense, this work may be regarded as a
continuation and improvement of the results obtained in
\cite{EpsLe00}.

The paper is organized as follows. Section \secref{sec.groupoids}
is devoted to a brief introduction to groupoids and Lie groupoids;
in particular, we define the normalizoid of a subgroupoid within a
groupoid, which is just the generalization of the notion of
normalizer in the context of groups. An important family of
examples is provided by the frame-groupoid, consisting of all the
linear isomorphisms between the tangent spaces at all the points
of a manifold $M$; if $M$ is equipped with a Riemannian metric
$g$, one can introduce the notion of orthonormal groupoid (taking
the orthogonal part of the linear isomorphisms given by the polar
decomposition). If, without necessarily possessing a distinguished
Riemannian metric, $M$ is endowed with a volume form, one obtains
the Lie subgroupoid of unimodular isomorphisms. In Section
\secref{sec.gstructures} we analyze the relations between Lie
groupoids and principal bundles; in particular, we examine the
relation between the frame groupoid and $G$-structures on a
manifold $M$. In Section \secref{sec.constitutive.equation} we
study the concepts of material symmetry and material symmetry
groups, and in Section \secref{sec.uniformity} we discuss
uniformity and homogeneity. Finally, Section
\secref{sec.unisymmetry} is devoted to study the case of FGM
materials, and the geometric characterization of homogeneity in
this case is obtained for both solid and fluids.


\section{Groupoids} \label{sec.groupoids}

Groupoids are a generalization of groups; indeed, they have a
composition law with respect to which there are some identity
elements and every element has an inverse. For a good reference on
groupoids, the reader is refered to Mackenzie \cite{Mack87}.

\begin{defn} \label{def.groupoid}
Given two sets $\Omega$ and $M$, a \emph{groupoid $\Omega$ over
$M$}, the \emph{base}, consists of these two sets together with
two mappings $\alpha, \beta:\Omega\rightarrow M$, called the
\emph{source} and the \emph{target projections}, and a composition
law satisfying the following conditions:
\begin{enumerate}
\item The composition law is defined only for those
$\eta,\xi\in\Omega$ such that $\alpha(\eta)=\beta(\xi)$ and, in
this case, $\alpha(\eta\xi)=\alpha(\xi)$ and
$\beta(\eta\xi)=\beta(\eta)$. We will denote
$\Omega_\Delta\subset\Omega\times\Omega$ the set of such pairs of
elements. \item The composition law is associative, that is
$\zeta(\eta\xi)=(\zeta\eta)\xi$ for those
$\zeta,\eta,\xi\in\Omega$ such that each member of the
previous equality is well defined. \item For each $x\in M$ there
exists an element $1_x\in\Omega$, called the \emph{unity over
$x$}, such that
\begin{enumerate}
\item $\alpha(1_x)=\beta(1_x)=x$; \item $\eta\cdot1_x=\eta$,
whenever $\alpha(\eta)=x$; \item $1_x\cdot\xi=\xi$, whenever
$\beta(\xi)=x$.
\end{enumerate}
\item For each $\xi\in\Omega$ there exists an element
$\xi^{-1}\in\Omega$, called the \emph{inverse of $\xi$}, such that
\begin{enumerate}
\item $\alpha(\xi^{-1})=\beta(\xi)$ and
$\beta(\xi^{-1})=\alpha(\xi)$; \item $\xi^{-1}\xi=1_{\alpha(\xi)}$
and $\xi\xi^{-1}=1_{\beta(\xi)}$.
\end{enumerate}
\end{enumerate}
The groupoid $\Omega$ will be said \emph{transitive} if, for every
pair $x,y\in M$, the set of elements that have $x$ as source and
$y$ as target, \ie\
$\Omega_{x,y}=\alpha^{-1}(x)\cap\beta^{-1}(y)$, is not empty.
\par A subset $\Omega'\subset\Omega$ is said to be a \emph{subgroupoid
of $\Omega$ over $M$} if itself is a groupoid over $M$ with the
composition law of $\Omega$.
\end{defn}

\begin{figure}[h]
\centering
\includegraphics[scale=0.6]{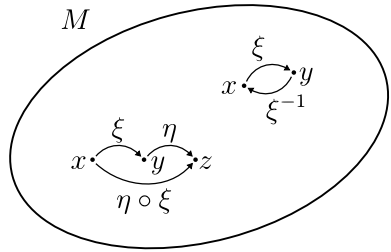}
\caption{The arrow picture.} \label{fig.groupoid}
\end{figure}

The elements of $M$ are often called \emph{objects} and those of
$\Omega$ \emph{arrows} due to their graphical interpretation as we
may see in the Figure \ref{fig.groupoid} or in the example
\ref{eg.trivial.groupoid}. By the very definition of groupoids,
the unity over an object and the inverse of an arrow are unique.
Note also that $\Omega_{x,x}$ is a group and the unity $1_x$ is
the group identity.

\begin{expl}[The trivial groupoid] \label{eg.trivial.groupoid}
Let $M$ denote any non-empty set. The Cartesian product $M\times
M$ is trivially a groupoid over $M$. The source of an arrow
$(x,y)$ is $x$ and the target $y$, and the composition
$(y',z)\cdot(x,y)$ is $(x,z)$ if and only if $y'=y$.
\end{expl}

\begin{expl}[The action groupoid] \label{eg.action.groupoid}
Now, let $G$ be a group acting on the left on $M$. Then the
product $G\times M$ is a groupoid over $M$ with the following
structural maps:
\begin{itemize}
\item the source, $\alpha(g,x)=x$; \item the target,
$\beta(g,x)=g\cdot x$; \item and the composition law,
$(h,y)\cdot(g,y)=(h\cdot g,x)$ if and only if $y=g\cdot x$.
\end{itemize}
With these considerations, the unity over an element $x\in M$ and
the inverse of an arrow $(g,x)\in G\times M$ are respectively
given by $1_x=(e,x)$ and $(g^{-1},g\cdot x)$, where $e\in G$
denotes the identity and $g^{-1}$ the inverse of $g$.
\end{expl}

\begin{prop} \label{th.groupoid.conjugation}
Let $\Omega$ be a groupoid over a set $M$. Then, given three
points $x,y,z\in M$ such that they can be connected by arrows, we
have the relation
\begin{equation} \label{eq.groupoid.transformation}
\Omega_{x,z} = g \cdot \Omega_{x,y} = \Omega_{y,z} \cdot f, \quad
\forall g\in\Omega_{y,z},\ \forall f\in \Omega_{x,y};
\end{equation}
in particular,
\begin{equation} \label{eq.groupoid.conjugation}
\Omega_{y,y} = g \cdot \Omega_{x,x} \cdot g^{-1}, \quad \forall
g\in\Omega_{x,y}.
\end{equation}
\end{prop}

For the moment, we have only algebraic structures on groupoids.
Let us endow them with differential structures.

\begin{defn} \label{def.groupoid.lie}
We say that a groupoid $\Omega$ over $M$ is a \emph{differential
groupoid} if the groupoid $\Omega$ and the base $M$ are equipped
with respective differential structures such that:
\begin{enumerate}
\item \label{def.groupoid.lie.projection}  the source and the
target projections $\alpha,\beta:\Omega\rightarrow M$ are smooth
surjective submersions; \item \label{def.groupoid.lie.inclusion}
the \emph{unity} or \emph{inclusion map} $i:x\in M\mapsto
1_x\in\Omega$ is smooth; \item
\label{def.groupoid.lie.composition} and the composition law,
defined on $\Omega_\Delta$, is smooth.
\end{enumerate}
Additionally if $\Omega$ is transitive, then we call it a
\emph{Lie groupoid}.
\par A subgroupoid $\Omega'$ of a differential (or Lie) groupoid
$\Omega$ which is in turn a differential groupoid with the
restricted differential structure is called a \emph{differential
subgroupoid} (resp. Lie subgroupoid).
\end{defn}

Note that the condition \itref{def.groupoid.lie.projection} in
Definition \ref{def.groupoid.lie} implies that the
$\alpha\beta$-diagonal $\Omega_\Delta$ is an embeded submanifold
of $\Omega\times\Omega$, and then
\itref{def.groupoid.lie.composition} makes sense. Ver Eecke showed
(\cf\ \cite{Mack87}) that, even with more relaxed conditions, the
inverse map $\xi\in\Omega\mapsto\xi^{-1}\in\Omega$ is smooth, and
therefore a diffeomorphism. In fact, there is a more general way
to define groupoids and subgroupoids (differentiable or not) as
the reader may find in \cite{Mack87}, but for our purposes these
definitions will be sufficient.

\begin{expl}[The frame groupoid] \label{eg.frame.groupoid}
Let $M$ be a smooth manifold with dimension $n$ and consider the
space of linear isomorphisms between tangent spaces to $M$ at any
pair of points, namely
\begin{equation} \label{eq.frame.groupoid}
\Pi(M) = \bigcup_{x,y\in M}\Iso(T_xM,T_yM).
\end{equation}
This set is called the \emph{frame groupoid} of $M$ and, in fact,
it is a Lie groupoid over $M$, as we are going to show.
\par First of all, we must give a manifold structure to $\Pi(M)$. Let
$(U,\phi)$ and $(V,\psi)$ be two charts of $M$ and consider the
map given by
\begin{equation} \label{eq.frame.groupoid.chart}
\begin{array}{rcl}
\chi:W &\To&     \phi(U)\times\Gl(n)\times\psi(V)\\
     A &\Mapsto& (x^i,A_i^j,y^j)
\end{array}
\end{equation}
where $\Gl(n)$ denotes the general linear group on $\Rn$,
\begin{equation} \label{eq.frame.groupoid.coordinates}
W=\bigcup_{x\in U,y\in V}\Iso(T_xM,T_yM) \quad\textrm{and}\quad
A\prth{\frac\partial{\partial x^i}}=A_i^j\frac\partial{\partial
y^j}.
\end{equation}
 By means of the induced chart $(W,\chi)$ we endow $\Pi(M)$ with a
differential structure of dimension $2n+n^2$.
\par The structural maps are given in the following way:
\begin{itemize}
\item the source and the target projections: if
$A\in\Iso(T_xM,T_yM)$, then $\alpha(A)=x$ and $\beta(A)=y$; \item
the composition law is the natural composition between
isomorphisms when it is defined; \item and the inclusion: if $x\in
M$, then the unity $1_x$ over $x$ is the identity map of
$\Gl(T_xM)=\Iso(T_xM,T_xM)$.
\end{itemize}
These maps define clearly a groupoid over $M$ and, through
\eqref{eq.frame.groupoid.chart} and
\eqref{eq.frame.groupoid.coordinates}, they are smooth for the
differential structure naturally induced from the one of $M$.
\end{expl}

\begin{expl}[The unimodular groupoid] \label{eg.unimodular.groupoid}
Let $M$ be an orientable smooth manifold of dimension $n$ and let
$\rho$ be a volume form on it (in a more general case, without the
assumption of orientation, we can consider a volume density). We
can use $\rho$ to define a determinant function over the frame
groupoid $\Pi(M)$ by the formula:
\begin{equation} \label{eq.determinant}
\rho(A\cdot v_1,\dots,A\cdot v_n) =
{\det}_\rho(A)\cdot\rho(v_1,\dots,v_n) \quad \forall A\in\Pi(M),
\end{equation}
where $v_1,\dots,v_n\in T_{\alpha(A)}M$. Now, it is easy to check
that the set of unimodular transformations
\begin{equation} \label{eq.unimodular.groupoid}
\UU(M) = {\det}_\rho^{-1}(\set{-1,+1}),
\end{equation}
which is called the \emph{unimodular groupoid}, is a transitive
subgroupoid of $\Pi(M)$. In fact, it is a Lie subgroupoid of
$\Pi(M)$, since $\det_\rho$ is a smooth submersion and thus
$\UU(M)$ is a closed submanifold.
\end{expl}

\begin{expl}[The orthogonal groupoid] \label{eg.orthogonal.groupoid}
Let $(M,g)$ be a Riemannian manifold of dimension $n$ and consider
the space of orthogonal linear isomorphisms between tangent spaces
to $M$ at any pair of points, namely
\begin{equation} \label{eq.orthogonal.groupoid}
\OO(M) = \bigcup_{x,y\in M}O(T_xM,T_yM).
\end{equation}
This set is called the \emph{orthogonal groupoid} of $M$ and, with
the restriction to it of the structure maps of the frame groupoid
$\Pi(M)$, $\OO(M)$ is a subgroupoid of $\Pi(M)$. Since $\OO(M)$ is
defined by closed and smooth conditions, namely
$$ \OO(M) = \set{A\in\Pi(M)\ :\ A^{-1}=A^T}, $$
this set is a closed submanifold of $\Pi(M)$, and thus a Lie
subgroupoid.
\par Furthermore, the orthogonal groupoid $\OO(M)$ is also a Lie
subgroupoid of the unimodular groupoid $\UU(M)$ related to the
Riemannian density induced by the metric.
\end{expl}

\begin{defn} \label{def.normalizoid}
Let $\Omega$ be a groupoid over $M$; then the \emph{normalizoid}
of a subgroupoid $\tilde\Omega$ of $\Omega$ is the set defined by
\begin{equation} \label{eq.normalizoid}
N(\tilde\Omega) = \set{g\in\Omega_{x,y}\ :\ \tilde\Omega_y =
g\cdot\tilde\Omega_x\cdot g^{-1},\ x,y\in B}.
\end{equation}
\end{defn}

From the definition, it is obvious that a subgroupoid
$\tilde\Omega$ of a groupoid $\Omega$ is also a subgroupoid of its
normalizoid $N(\tilde\Omega)$ which is, in turn, a subgroupoid of
the ambient groupoid $\Omega$.

Note that the group over a base point in the normalizoid is the
normalizer of the group over this point in the subgroupoid, that
is
\begin{equation} \label{eq.normalizer}
(N(\tilde\Omega))_{x,x} = N(\tilde\Omega_{x,x}),
\end{equation}
which explains the used terminology. The difference between a
subgroupoid and its normalizoid can be huge. For instance, given a
transitive groupoid $\Omega$ over a set $M$, consider its base
groupoid, that is the subgroupoid consisting of the groupoid
unities:
\begin{equation} \label{eq.base.groupoid}
1(\Omega) = \set{1_x : x\in M}.
\end{equation}
Then, the normalizoid of $1(\Omega)$ in $\Omega$ is the whole
groupoid $\Omega$. From now on, we will focus on subgroupoids of
the frame groupoid over a manifold and we will see how to reduce
the normalizoid of a subgroupoid whenever an extra structure is
avaible on the base manifold.

First of all, recall that there exists a unique decomposition of a
linear isomorphism into an orthogonal part and a symmetric one.
More precisely, let $F:E\To E'$ be a linear isomorphism between
two inner product vector spaces $E$ and $E'$. There exist an
orthogonal map $R:E\To E'$ and positive definite symmetric maps
$U:E\To E$, $V:E'\To E'$ such that:
\begin{equation} \label{eq.polar.decomposition}
F=R\cdot U \and F=V\cdot R.
\end{equation}
As we have mentioned, each of these decompositions is unique and
they are called the left and right polar decompositions of $F$,
respectively; the orthogonal part $R$ will be denoted by
$F^\perp$.

\begin{prop} \label{th.orthogonal.reduced.groupoid}
Let $\Omega$ be a (transitive) subgroupoid of the frame groupoid
$\Pi(M)$ of a Riemannian manifold $(M,g)$. Denote by $\bar\Omega$
the set of the orthogonal part of elements of $\Omega$, that is
\begin{equation} \label{eq.orthogonal.reduced.groupoid}
\bar\Omega = \set{ F^\perp\ :\ F\in\Omega }.
\end{equation}
Then $\bar\Omega$ is a (transitive) subgroupoid of the orthogonal
groupoid $\OO(M)$. We call $\bar\Omega$ the \emph{orthogonal
reduction of $\Omega$} (or the \emph{reduced groupoid}, for the
sake of simplicity).
\end{prop}

\begin{proof}
In order to show that $\bar\Omega$ is a subgroupoid of $\OO(M)$,
we only have to check that it is a groupoid over $M$ with the
restriction of the structure maps of $\Pi(M)$, which is clear once
we note that for any three linear isomorphisms $F_1,F_2,F_3$, such
that $F_3=F_2\cdot F_1$, we have by the uniqueness of the polar
decomposition that $F_3^\perp=F_2^\perp\cdot F_1^\perp$.
\end{proof}

Note that the orthogonal reduction of a normalizoid is not
necessarily a subgroupoid of the original one.

\begin{prop} \label{th.orthogonal.reduced.normaloid}
In the hypotesis of Proposition
\ref{th.orthogonal.reduced.groupoid}, if $\Omega$ is such that,
for every base point $x\in M$, $\Omega_{x,x}$ is a subgroup of
$\OO_{x,x}(M)$ (the orthogonal group at $x$), then the orthogonal
reduction of the normalizoid of $\Omega$ coincides with the
intersection of the orthogonal groupoid and the normalizoid
itself, \ie
\begin{equation} \label{eq.orthogonal.reduced.normaloid}
\bar\NN(\Omega) = \NN(\Omega)\cap\OO(M).
\end{equation}
\end{prop}

\begin{proof}
The inclusion $\bar\NN(\Omega)\supset\NN(\Omega)\cap\OO(M)$ is
clear and, from the above Proposition
\ref{th.orthogonal.reduced.groupoid}, we have
$\bar\NN(\Omega)\subset\OO(M)$, thus we only need to show that
$\bar\NN(\Omega)\subset\NN(\Omega)$. Let
$R\in\bar\NN_{x,y}(\Omega)$, then there exist a linear isomorphism
$F\in\NN_{x,y}(\Omega)$ such that $F^\perp=R$. Since $F$
conjugates the orthogonal subgroups $\Omega_{x,x}$ and
$\Omega_{y,y}$, so does its orthogonal part (\cf\ \cite{EpsLe00},
Lemma A.2). Hence, $R\in\NN_{x,y}(\Omega)$ and $\bar\NN(\Omega)
\subset \NN(\Omega)\cap\OO(M)$.
\end{proof}

Similar results can be given whenever $M$ is equipped with a
volume form.

\begin{prop} \label{th.unimodular.reduced.groupoid}
Given a smooth manifold $M$, suppose it is endowed with a volume
form (or density) $\rho$. If $\Omega$ denotes a (transitive)
subgroupoid of the frame groupoid $\Pi(M)$, then the set
\begin{equation} \label{eq.unimodular.reduced.groupoid}
\Omega^1 = \Omega/{\det}_\rho,
\end{equation}
is a (transitive) subgroupoid of the unimodular groupoid $\UU(M)$
associated with $\rho$ and it will be called the \emph{unimodular
reduction} of $\Omega$.
\par Even more, if $\Omega$ is such that, for every base point $x\in
M$, $\Omega_{x,x}$ is a subgroup of $\UU_{x,x}(M)$ (the unimodular
group at $x$), then the unimodular reduction of the normalizoid of
$\Omega$ coincides with the intersection of the unimodular
groupoid and the normalizoid itself, \ie
\begin{equation} \label{eq.unimodular.reduced.normaloid}
\NN^1(\Omega) = \NN(\Omega)\cap\UU(M).
\end{equation}
\end{prop}

\section{$G$-structures} \label{sec.gstructures}

Lie subgroupoids of the frame groupoid of a manifold are closely
related to another geometric object: $G$-structures, which are a
particular case of fiber bundles. For a comprehensive reference
related to principal fiber bundles and $G$-structures see
\cite{Fuji72,KoNo69A,KoNo69B}. We give here their definition and
some results about the interconnection with groupoids.

\begin{defn} \label{def.principal.bundle}
Given two manifolds $P,M$ and a Lie group $G$, we say that $P$ is
a \emph{principal bundle} over $M$ with \emph{structure group} $G$
if $G$ acts on the right on $P$ and the following conditions are
satisfied:
\begin{enumerate}
\item \label{def.principal.bundle.free.action} the action of $G$
is free, \ie\ the fact that $ua=u$ for some $u\in P$ implies
$a=e$, the identity element of $G$; \item
\label{def.principal.bundle.cocient} $M=P/G$, which implies that
the canonical projection $\pi:P\To M$ is differentiable; \item
\label{def.principal.bundle.triviallity} $P$ is locally trivial,
\ie\ $P$ is locally isomorphic to the product $M\times G$, which
means that for each point $x\in M$ there exists an open
neighborhood $U$ and a diffeomorphism $\Phi:\pi^{-1}(U)\To U\times
G$ such that $\Phi=\pi\times\phi$, where the map
$\phi:\pi^{-1}(U)\To G$ has the property $\phi(ua)=\phi(u)a$ for
all $u\in\pi^{-1}(U)$, $a\in G$.
\end{enumerate}
A principal bundle is commonly denoted by $P(M,G)$, $\pi:P\To M$
or simply by $P$, when there is no ambiguity. The manifold $P$ is
called the \emph{total space}, $M$ the \emph{base space}, $G$ the
\emph{structure group} and $\pi$ the \emph{projection}. The closed
submanifold $\pi^{-1}(x)$, with $x\in M$, is called the
\emph{fiber over $x$} and is denoted $P_x$; if $u\in P$,
$P_{\pi(u)}$ is called the \emph{fiber through $u$} and is denoted
$P_u$. The maps given in \itref{def.principal.bundle.triviallity}
are called \emph{(local) trivializations}.
\par It should be remarked that a similar definition can be given for
left principal bundles using left actions.
\end{defn}

Notice that any fiber $P_x$ is diffeomorphic to the structure
group $G$, but not canonically so. On the other hand, if we fix
$u\in P_x$, then $P_u=uG$. We may visualize a principal fiber
bundle $P(M,G)$ as a copy of the structure Lie group $G$ at each
point of the base manifold $M$ in a diffentiable way as it is
stated by the trivialization property
\itref{def.principal.bundle.triviallity}.

An elementary example of principal bundle is the \emph{frame
bundle} $\FM$ of a manifold $M$. This manifold consists of all the
reference frames at all the point of $M$. The frame bundle $\FM$
is a principal bundle over $M$ with structure group $\Gl(n)$,
where $n$ is the dimension of $M$. As it is obvious, the canonical
projection $\pi$ sends any frame $x\in\FM$ to the base point $x\in
M$ where it lies. The right action of $\Gl(n)$ over $M$ is defined
in the following way:
\begin{equation} \label{eq.group.action}
\begin{array}{rcl}
R:\FM\times\Gl(n) &\To&     \FM\\
            (z,a) &\Mapsto& R_az=z\cdot a=(a_i^jv_j),
\end{array}
\end{equation}
where $(a_i^j)$ is the matrix representation of $a\in\Gl(n)$ in
the canonical basis of $\Rn$ and $(v_i)$ is the ordered basis
given by $z\in\FM$.

\begin{defn} \label{def.principal.reduction}
Let $P(M,G)$ and $Q(M,H)$ be two principal bundles such that $Q$
is an embedded submanifold of $P$ and $H$ is a Lie subgroup of
$G$. We say that $Q(M,H)$ is a \emph{reduction} of the structure
group $G$ of $P$ if the principal bundle structure of $Q(M,H)$
comes from the restriction of the action of $G$ on $P$ to $H$ and
$Q$. In this case, we call $Q$ the \emph{reduced bundle}.
\end{defn}

Consider the following (non rigorous) construction: take a
principal bundle $P(M,G)$, shrink its structure group to a Lie
subgroup $H$ of $G$, fix an element $u\in P$ in each fibre of the
bundle and apply the action of $H$ to each of these chosen
elements; this gives us a subset $Q\subset P$. The obtained set
$Q$ is a reduced bundle when the selection of the $u$'s is made
smoothly and with certain compatibility.

\begin{defn} \label{def.gstructure}
Let $M$ be an $n$-dimensional smooth manifold and $G$ a Lie
subgroup of $\Gl(n)$; then a \emph{$G$-structure} $G(M)$ is a
$G$-reduction of the frame bundle $\FM$.
\end{defn}

Note that there may exist different $G$-structures with the same
structure group. As an example of $G$-structure, consider a
Riemannian manifold $(M,g)$. The set of orthonormal references of
$\FM$ gives us an $O(n)$-structure. In fact, any $O(n)$-structure
on $M$ is equivalent to a Riemannian structure (see
\cite{Fuji72}).

Now let us introduce two results from \cite{LeMar04} that show how
a $G$-structure may arise from a Lie groupoid.

\begin{prop} \label{th.groupoid.principal.bundle}
Let $\Omega$ be a Lie groupoid over a smooth manifold $M$ with
source and target projections $\alpha$ and $\beta$, respectively.
Given any point $x\in M$, we have that:
\begin{enumerate}
\item $\Omega_{x,x}=\alpha^{-1}(x)\cap\beta^{-1}(x)$ is a Lie
group and \item $\Omega_x=\alpha^{-1}(x)$ is a principal
$\Omega_{x,x}$-bundle over $M$ whose canonical projection is the
restriction of $\beta$.
\end{enumerate}
\end{prop}

Given a smooth manifold $M$ of dimension $n$, any reference
$z\in\FM$ (at a point $x\in M$) may be seen as the linear mapping
$e_i\in\Rn\mapsto v_i\in T_xM$, where $(e_1,\dots,e_n)$ is the
canonical basis of $\Rn$ and $(v_1,\dots,v_n)$ the basis of $T_xM$
defined by $z$.

\begin{thm} \label{th.groupoid.gstructure}
Suppose that $M$ is a smooth $n$-dimensional manifold and $\Omega$
is a Lie subgroupoid of the frame groupoid $\Pi(M)$. If $\alpha$
and $\beta$ denote the respective source and target projections of
$\Omega$, then we have that for any point $x\in M$ and any frame
reference $z\in\FM$ at $x$:
\begin{enumerate}
\item $G_z=z^{-1}\cdot\Omega_{x,x}\cdot z$ is a Lie subgroup of
$\Gl(n)$ and \item the set $\Omega_z$ of all the linear frames
obtained by translating $z$ by $\Omega_x$, that is
\begin{equation} \label{eq.groupoid.gstructure}
\Omega_z = \set{g_{x,y}\cdot z\ :\ g_{x,y}\in\Omega_x},
\end{equation}
is a $G_z$-structure on $M$.
\end{enumerate}
Once the reference $z$ is fixed, the linear frames that lie in the
$G_z$-structure are called \emph{adapted} or \emph{distinguished}
references.
\end{thm}

Even though the frame groupoid (and hence each of its
subgroupoids) acts on the left on the base manifold, the
structural group that arises from a frame subgroupoid acts
naturally on the right on any of the induced $G$-structures:
\begin{equation} \label{eq.groupoid.gstructure.action}
z_y\cdot g_{z_x} = (g_{x,y}\cdot z_x)\cdot(z_x^{-1}\cdot
g_{x,x}\cdot z_x) = g_{x,y}\cdot g_{x,x}\cdot z_x = g'_{x,y}\cdot
z_x = z'_y,
\end{equation}
where $z_x\in\FF_xM$, $z_y\in(\Omega_{z_x})_y$, $g_{z_x}\in
G_{z_x}$, $g_{x,y}\in\Omega_{x,y}$ and so on.

\begin{rem} \label{rmk.groupoid.gstructure.unicity}
It is readily seen from equation \eqref{eq.groupoid.conjugation}
that two $G$-structures that come from the same Lie groupoid are
equal if and only if they have a reference in common,
\begin{equation} \label{eq.groupoid.gstructure.equivalence}
\Omega_{z_1}=\Omega_{z_2} \Leftrightarrow
\Omega_{z_1}\cap\Omega_{z_2}\neq\emptyset.
\end{equation}
Here ``equal'' means that the two $G$-structures are the same as
sets and they have the same structure groups. By the above
statement, given two $G$-structures $\Omega_{z_1}$ and
$\Omega_{z_2}$ induced by a Lie groupoid $\Omega$, we can suppose
without loss of generality that $z_1$ and $z_2$ are linear frames
at the same base point. Thus, it is easy to see that their
respective structure groups $G_{z_1}$ and $G_{z_2}$ are conjugate;
more precisely:
\begin{equation} \label{eq.groupoid.gstructure.conjugation}
G_{z_2} = z_2^{-1}z_1 \cdot G_{z_1} \cdot z_1^{-1}z_2.
\end{equation}
In short, given a Lie subgroupoid $\Omega$ of $\Pi(M)$, the frame
bundle $\FM$ is the disjoint union of $G$-structures related to
$\Omega$ by Theorem \ref{th.groupoid.gstructure}. Moreover, they
have conjugate group structures and one of these $G$-structures
may be transformed to another by mean of any element $g\in\Gl(n)$
that conjugates their structural groups. Hence, modulo these
transformations, a $G$-structure related to a Lie subgroupoid
$\Omega$ of $\Pi(M)$ is unique, which is clear since $\Omega$ is
fixed.
\end{rem}

A natural question is whether Theorem \ref{th.groupoid.gstructure}
has a converse. Given a $G$-structure, it seems reasonable to be
able to choose differentially isomorphisms that transform adapted
references to their counterparts.

\begin{thm} \label{th.gstructure.groupoid}
Let $\omega$ be a $G$-structure over an $n$-dimensional smooth
manifold $M$. Then the set of linear isomorphism that transforms
distinguished frames into distinguished frames, that is the set
\begin{equation} \label{eq.gstructure.groupoid}
\Omega = \set{ A\in\Pi(M)\ :\ Az\in\omega,\ z\in\omega_{\alpha(A)}
},
\end{equation}
where $\Pi(M)$ is the frame groupoid of $M$ and $\alpha$ the
source projection, is a Lie soubgroupoid of $\Pi(M)$. Furthermore,
for any reference frame $z\in\omega$, the $G$-structure associated
to $\Omega$ and given by Theorem \ref{th.groupoid.gstructure}
coincides with $\omega$, \ie
\begin{equation} \label{eq.gstructure.groupoid.equivalence}
\Omega_z = \omega \and G_z=G.
\end{equation}
\end{thm}

\begin{proof}
The set defined by equation \eqref{eq.gstructure.groupoid} is
obviously a transitive subgroupoid of $\Pi(M)$. It remains only to
show that it is a differential groupoid with the restriction of
the structural maps. Given two local cross-sections $(U,\sigma)$
and $(V,\tau)$ of $\omega$, consider the set of isomorphisms in
$\Omega$ with source in $U$ and target in $V$, namely
\begin{equation} \label{eq.gstructure.groupoid.chart.domain}
\Omega_{U,V} = \alpha^{-1}(U)\cap\beta^{-1}(V),
\end{equation}
where $\alpha$ and $\beta$ are the restrictions to $\Omega$ of the
source and the target projections of $\Pi(M)$. Given an
isomorphism $A\in\Omega_{U,V}$, let $x=\alpha(A)\in U$ and
$y=\beta(A)\in V$. If we denote the components of the ordered
bases $\sigma(x)$ and $\tau(y)$ by $(\sigma_i(x))$ and
$(\tau_j(y))$ respectively, we have that there exist coefficients
$A_i^j$ such that
\begin{equation} \label{eq.gstructure.groupoid.frame.transformation}
A\sigma_i(x) = A_i^j\tau_j(y).
\end{equation}
Since $\sigma(x)=(\sigma_i(x))$ is a linear frame at $x$ in
$\omega$, $A\sigma(x)=(A_i^j\tau_j(y))$ is a linear frame at $y$
in $\omega$ too. But $\tau(y)=(\tau_j(y))$ is also a linear frame
at $y$ in $\omega$, thus $a=(A_i^j)$ must necessarily be an
element of the structure group $G$. This consideration being made,
we define the coordinate chart $\Phi_{\sigma,\tau}$ by
\begin{equation} \label{eq.gstructure.groupoid.chart.coordinates}
\begin{array}{rcl}
\Phi_{\sigma,\tau}:\Omega_{U,V} &\To&     U\times G\times V\\
                              A &\Mapsto& (x,a,y)
\end{array}.
\end{equation}
Given a covering of $M$ by local sections of $\omega$, say
$\Sigma$, the atlas
\begin{equation} \label{eq.gstructure.groupoid.chart.atlas}
\set{(\Omega_{U,V},\Phi_{\sigma,\tau})\ :\
(U,\sigma),(V,\tau)\in\Sigma}
\end{equation}
defines a smooth structure on $\Omega$, from which it is a
straightforward computation to show that the projections $\alpha$
and $\beta$ and the composition law are smooth.
\end{proof}

\begin{rem} \label{rmk.gstructure.groupoid}
The result we have just proved, toghether with Theorem
\ref{th.groupoid.gstructure}, shows the equivalence between Lie
subgroupoids of $\Pi(M)$ and reductions of the frame bundle $\FM$.
In fact it is still true for principal bundles in general: by
Proposition \ref{th.groupoid.principal.bundle} we are able to
associate some principal bundles to a groupoid and, given a
principal bundle $P(M,G)$, the set of maps $\phi_{x,y}:P_x\To P_y$
such that $\phi_{x,y}(u\cdot g)=\phi_{x,y}(u)\cdot\phi(g)$, for a
suitable group isomorphism $\phi:G\To G$, is a Lie groupoid
related to $P$ by Proposition \ref{th.groupoid.principal.bundle}.
\end{rem}

\begin{defn} \label{def.integrable.gstructure}
A $G$-structure $G(M)$ over a manifold $M$ is said to be
\emph{integrable} if there exists an atlas
$\set{(U_\alpha,\phi_\alpha)}_{\alpha\in A}$ of the base manifold,
such that the induced cross-sections
$\sigma_\alpha(x)=(T_x\phi_\alpha)^{-1}$ take values in $G(M)$.
\end{defn}

By the very definition, if a $G$-structure is integrable, the same
happens to all its conjugate $G$-structures.

\begin{thm} \label{th.integrable.gstructure}
A $G$-structure over a manifold $M$ with dimension $n$ is
integrable if and only if it is locally isomorphic to the standard
$G$-structure of $\Rn$, that is, to $\Rn\times G$.
\end{thm}

The next result will be useful in the next section.

\begin{lem} \label{th.intersection.groupoid}
Let $M$ be a manifold. If $\Omega$ and $\tilde\Omega$ are two
subgroupoids of the frame groupoid $\Pi(M)$, then their
intersection $\hat\Omega:=\Omega\cap\tilde\Omega$ is again a
subgroupoid of $\Pi(M)$ (and of $\Omega$ and $\tilde\Omega$).
Furthermore, if they are Lie groupoids, then we have the following
relations:
\begin{equation} \label{eq.intersection.groupoid}
\hat\Omega_z=\Omega_z\cap\tilde\Omega_z \quad\textrm{and}\quad
\hat G_z= G_z\cap\tilde G_z,
\end{equation}
where $z\in\FM$ a is fixed frame and $\Omega_z$, $\tilde\Omega_z$,
$\hat\Omega_z$, $G_z$, $\tilde G_z$ and $\hat G_z$ are the
respective $G$-structures and structural groups.
\end{lem}


\section{The Constitutive Equation} \label{sec.constitutive.equation}

In the most general sense (see \cite{MarsHu83}, for instance), a
\emph{body} is a manifold $B$ that can be embedded in a Riemannian
manifold $(S,g)$ with the same dimension, the \emph{ambient
space}. Usually, the body $B$ is a simply connected open set of
$\RR^3$ and the ambient space is $\RR^3$ itself with the standard
metric. Each embedding $K:B\rightarrow S$ is called a
\emph{configuration} and its tangent map $TK:TB\rightarrow TS$ is
called an \emph{infinitesimal configuration}. If we fix a
configuration $K$ (the \emph{reference configuration}) and  we
pick an arbitrary configuration $\tilde K$, then the embedding
compositon $\phi=\tilde K\circ K^{-1}:K(B)\subset S\rightarrow S$
is considered as a body \emph{deformation}  and we call its
tangent map $T_X\phi$ at a point $X$ in $B$ an \emph{infinitesimal
deformation} or the \emph{deformation gradient}, usually denoted
by $F$. Since $(S,g)$ is a Riemannian manifold, we can induce a
Riemannian metric on $B$ by the pull-back of $g$ by a reference
configuration $K$. Since the metric on $B$ depends from a chosen
reference configuration, it is not canonical. However, for solid
materials, we are able to define an ``almost'' unique metric
compatible with the material structure, as we will show in section
\secref{sec.uniform.elastic.solids}.

Usually, points in the body or in the reference configuration
(when they are identified) are denoted by capital letters $X$,
$Y$, $Z$, etc., and by small letters $x$, $y$, $z$, etc., in the
deformed configuration. At the moment we have the picture shown at
Figure \ref{fig.configuration}.

\begin{figure}[h]
\centering
\includegraphics[scale=0.6]{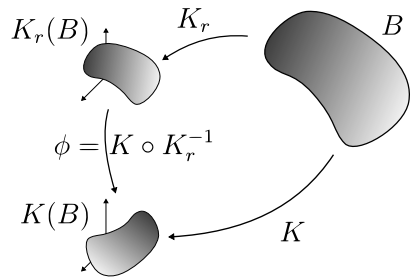}
\caption{Deformation in a reference configuration.}
\label{fig.configuration}
\end{figure}

As stated by \emph{the principle of determinism}, the mechanical
and thermal behaviors of a material or substance are determined by
a relation called \emph{the constitutive equation}. It does not
follow directly from physical laws but it is combined with other
equations that do represent physical laws (the conservation of
mass for instance) to solve some physical problems, like the flow
of a fluid in a pipe, or the response of a crystal to an electric
field. In our case of in\-te\-rest, \emph{elastic materials}, the
constitutive equation establishes that, in a given reference
configuration, \emph{the Cauchy stress tensor} depends only on the
material points and on the infinitesimal deformations applied on
them, that is
\begin{equation} \label{eq.constitutive.equation.elastic}
\sigma = \sigma(F_{K_r},K_r(X)).
\end{equation}
This relation is simplified in the particular case of
\emph{hyper\-elastic materials}, for which equation
\eqref{eq.constitutive.equation.elastic} becomes
\begin{equation} \label{eq.constitutive.equation.hyperelastic}
W = W(F_{K_r},K_r(X)).
\end{equation}
where $W$ is a scalar valued function which measures the stored
energy per unit volume.

Among other postulates (\emph{principle of determinism},
\emph{principle of local action}, \emph{principle of
frame-indifference}, etc.), it is claimed that a constitutive
equation must not depend on the reference configuration. It turns
out that equation \eqref{eq.constitutive.equation.elastic} (and
\eqref{eq.constitutive.equation.hyperelastic}) now can be written
in the form
\begin{equation} \label{eq.constitutive.equation}
\sigma = \sigma(F,X) \quad (W = W(F,X),\ \textrm{respectively)},
\end{equation}
where $F$ stands for the tangent map at $X$ of a local
con\-fi\-gu\-ra\-tion (deformation).

\begin{defn} \label{def.material.symmetry}
A \emph{material symmetry} at a given point $X\in B$ is a linear
isomorphism $P:T_XB\rightarrow T_XB$  such that
\begin{equation} \label{eq.material.symmetry}
\sigma(F\cdot P,X) = \sigma(F,X),
\end{equation}
for any deformation $F$ at $X$. The set of material symmetries at
$X\in B$ is denoted by $\GG(X)$ and it is called the
\emph{symmetry group} of $B$ at $X$. Given a configuration $K$, we
will denote by $\GG_K(X)$ the symmetry group $\GG(X)$ in the
configuration $K$, that is
\begin{equation} \label{eq.symmetry.group}
\GG_K(X) = T_XK\cdot\GG(X)\cdot(T_XK)^{-1}.
\end{equation}
\end{defn}

\begin{figure}[h]
\centering
\includegraphics[scale=0.6]{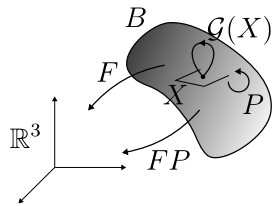}
\caption{Material symmetry.} \label{fig.material.symmetry}
\end{figure}

Different types of elastic materials are given in terms of their
symmetry groups. For instance, a point is solid whenever its
symmetry group in some reference configuration is a subgroup of
the orthogonal group $O(3)$ and, fluid whenever the orthogonal
group is a proper subgroup of the symmetry group. In
\cite{LeMar04,Wang67} it is possible to find a classification, due
to Lie, of the connected Lie subgroups of $\Sl(3)$ and their
corresponding Lie algebras.

\begin{defn} \label{def.elastic.material.point.classification}
Given an elastic material $B$, let $X\in B$ and consider its
symmetry group $\GG(X)$. If there exists a configuration $K$ such
that:
\begin{enumerate}
\item $\GG_K(X)$ is a subgroup of the orthogonal group of
transformations $O(3)$, then $X$ is said to be an \emph{elastic
solid point}. If furthermore
\begin{enumerate}
\item $\GG_K(X)=O(3)$, then we call $X$ a \emph{fully isotropic
elastic solid point}; \item $\GG_K(X)$ is a transverse orthogonal
group (a group of rotations which fix an axis), then $X$ is said
to be a \emph{transversely isotropic elastic solid point}; \item
$\GG_K(X)$ consists only of the identity element, then $X$ will be
a \emph{triclinic elastic solid point};
\end{enumerate}
\item $\GG_K(X)$ is a subgroup of the unimodular group of
transformations $U(3)$ and has the orthogonal group $O(3)$ as a
proper subgroup, then $X$ is said to be an \emph{elastic fluid
point}. If furthermore
\begin{enumerate}
\item $\GG_K(X)=\Sl(3)$ then we still call $X$ an \emph{elastic
fluid}; and \item $\GG_K(X)$ is a transverse unimodular group (a
group of unimodular transformations which fix an axis or a group
of unimodular transformations which fix a plane) then we call $X$
an \emph{elastic fluid crystal}.
\end{enumerate}
\end{enumerate}
The infinitesimal configuration $T_XK$ or the induced frame
$z=(T_XK)^{-1}$ is called an \emph{undistorted state} of $X$.
\end{defn}

This material classification is pointwise. A body is solid if
every point is solid.

\section{Uniformity and Homogeneity} \label{sec.uniformity}

To define the uniformity of a material, we first have to give a
criterion that  establishes when two points are made of the same
material. To compare their symmetry groups is not sufficient since
this is only a qualitative aspect. Indeed, consider two points in
a rubber band, one point may be relaxed while another point may be
under stress. But we are still able to release the stress on the
second point and bring it to the same state as the first one, and
then compare their responses.

\begin{defn} \label{def.material.isomorphism}
We say that two points $X,Y\in B$ are \emph{materially
isomorphic}, if there exists a linear isomorphism
$P_{XY}:T_XB\rightarrow T_YB$ such that
\begin{equation} \label{eq.material.isomorphism}
\sigma(F\cdot P_{XY},X) = \sigma(F,Y),
\end{equation}
for any deformation $F$ at $Y$. The linear map $P_{XY}$ is called
a \emph{material isomorphism}.
\end{defn}

\begin{figure}[h]
\centering
\includegraphics[scale=0.6]{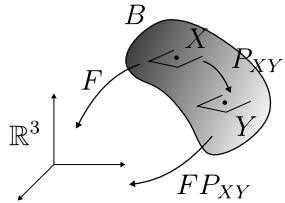}
\caption{Material isomorphism.} \label{fig.material.isomorphism}
\end{figure}

Even if the definition of material isomorphism and material
symmetries are mathematically similar, there is an important
conceptual difference. While the symmetry group of a point
characterizes the material behavior of that point, a material
isomorphism establishes a relation between two different points.
In fact, as already pointed out, a material symmetry can be viewed
as a material automorphism by identifying $X$ with $Y$ in the
above definition.

\begin{defn} \label{def.material.groupoid}
Given a material body $B$, the \emph{material groupoid} is the set
of all the material isomorphisms and symmetries, that is the set
\begin{equation} \label{eq.material.groupoid}
\GB = \set{ P\in\Pi(B)\textrm{ satisfying Definition
\ref{def.material.isomorphism}} }.
\end{equation}
\end{defn}

It is easy to check that the material groupoid $\GB$ is actually a
groupoid. Furthermore, it is a subgroupoid of the frame groupoid
$\Pi(B)$, but note that it is not necessarily a Lie groupoid or
even transitive as the frame groupoid. In fact, when all the
points of a body are pairwise related by a material isomorphism,
it means that the body consists only of one type of material. In
this case, it is materially uniform.

\begin{defn} \label{def.uniformity}
Given a material body $B$, we say that it is \emph{uniform} if the
material groupoid $\GB$ is transitive, and \emph{smoothly uniform}
when the material groupoid is a transitive differential groupoid
(and hence a Lie subgroupoid of $\Pi(B)$).
\end{defn}

A simple but important property of uniform materials is that the
groups of material symmetries are mutually conjugate by any
material isomorphism between the respective base points. To be
more precise, equation \eqref{eq.groupoid.conjugation} reads in
terms of elastic bodies:
\begin{equation} \label{eq.symmetry.group.conjugation}
\GG(Y) = P \cdot \GG(X) \cdot P^{-1}, \quad \forall
P\in\GG(B)_{X,Y},
\end{equation}
for any pair of materially isomorphic points $X,Y\in B$.

When we look a material through different configurations, there
are prefered states of the material we want to distinguish: \eg\
transversely isotropic solids have a fixed axis ``invariant''
under material isomorphisms that we prefer to align with the
vertical axis. Such a state may be modelized in an infinitesimal
configuration by a linear frame $z$. As we have just said, in the
material paradigm, this frame of reference $z$ has some behaviors
that will be mainted by material isomorphisms. If we consider the
set of all these distinguished references that arise from material
transformations of the `reference crystal' (see Figure
\ref{fig.reference.crystal}), then we obtain the so called
material $G$-structure of $B$. As far as we know, Wang was the
first to realize that the uniformity of a material can be modelled
by a $G$-structure \cite{Wang67}, although this fact was
emphasized by Bloom \cite{Bloom79}. For definiteness,

\begin{defn} \label{def.material.gstructure}
A \emph{material $G$-structure} of a smoothly uniform body $B$ is
any of the $G_z$-structures induced by the material groupoid
$\GG(B)$ as shown in Theorem \ref{th.groupoid.gstructure}. The
chosen frame of reference $z\in\FB$ is called the \emph{reference
crystal}.
\end{defn}

\begin{figure}[h]
\centering
\includegraphics[scale=0.6]{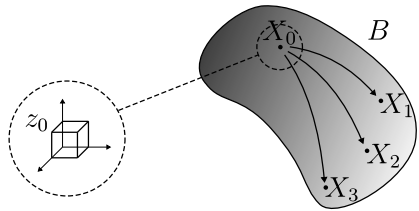}
\caption{The reference crystal.} \label{fig.reference.crystal}
\end{figure}

\begin{defn} \label{def.homogeneity}
Given a smoothly uniform body $B$, a configuration $K$ that
induces a cross-section of a material $G$-structure will be called
\emph{uniform}. If there exists an atlas
$\set{(U_\alpha,K_\alpha)}_{\alpha\in A}$ of $B$ of local uniform
configurations for a fixed material $G$-structure, the body $B$
will be said \emph{locally homogeneous}, and \emph{(globally)
homogeneous} if the body $B$ may be covered by just one uniform
configuration.
\end{defn}

The material concept of homogeneity corresponds to the
mathematical concept of integrability. By Theorem
\ref{th.integrable.gstructure}, a smoothly uniform body $B$ will
be locally homogenous if and only if one (and therefore any) of
the associated material $G$-structures is integrable. Let $K$ a
uniform configuration for a particular integrable $G$-structure
$G(B)$ of a homogeneous elastic material $B$. If $(X,v_1,v_2,v_3)$
denotes the cross section induced by $K$, thus the constitutive
equation \eqref{eq.constitutive.equation.elastic} may be written
in the form
\begin{equation} \label{eq.homogeneous.constitutive.equation}
\sigma = \sigma(F_K,K(X)) = \sigma(F^i_j,x^i),
\end{equation}
with obvious notation. Now note that, since through $K$ any
material isomorphism $P$ may be considered as an element of the
structure group $G$, which is clear for material symmetries, and
since the body $B$ is uniform, we have that
\begin{equation} \label{eq.homogeneous.constitutive.equation.proof}
\sigma(F^i_j,y^i) = \sigma(F_K,K(Y)) = \sigma(F_K\cdot P_K,K(X)) =
\sigma(F^i_k\cdot P^k_j,x^i) = \sigma(F^i_j,x^i).
\end{equation}
Thus, we have just proved the following result:

\begin{thm} \label{th.homogeneous.constitutive.equation.coordinates}
If $K$ is a uniform configuration of a homogeneous elastic body
$B$, the constitutive equation
\eqref{eq.constitutive.equation.elastic} is independent of the
material point and invariant under the right action of the structure
group $G$ of the $G$-structure $G(B)$ related to $K$. Thus,
\begin{equation} \label{eq.homogeneous.constitutive.equation.coordinates}
\sigma = \sigma(F^i_j) \quad\textrm{and}\quad \sigma(F^i_k\cdot
P^k_j) = \sigma(F^i_j)\ \textrm{for any}\ P\in G.
\end{equation}
\end{thm}

The physical interpretation of this theorem is that points of a
homogenous elastic body $B$ can be put by means of a configuration
$K$ in such a manner they are all at the same state, at least
locally. This configuration $K$ is uniform.

Even if the material $G$-structures of a smoothly uniform body $B$
are different (but equal via conjugation), there must be at least
one of them in which the structure group $G$ satisfies a condition
of the material classification
\ref{def.elastic.material.point.classification}.

\begin{defn} \label{def.elastic.material.uniform.classification}
Accordingly to Definition
\ref{def.elastic.material.point.classification}, a smoothly
uniform elastic body $B$ is \emph{solid} or \emph{fluid}, if all
the points are \emph{solid} or \emph{fluid}, respectively. Any of
the material $G$-structures for which the structure group fulfills
the classification is called \emph{undistorted}.
\end{defn}

\subsection{Uniform Elastic Solids} \label{sec.uniform.elastic.solids}

The following result is due to Wang (\cf\ \cite{Wang67}). In his
paper, Wang defines the material $G$-structures from the point of
view of atlases, families of cross-sections of the frame bundle,
instead of our approach through groupoids. These families are the
cross-sections of the resulting $G$-structures. When a material is
solid, it is possible to endow the body with a metric wich is
compatible with the material structure. Wang calls such a metric
an intrinsic metric.

\begin{thm} \label{th.uniform.elastic.solid.metric}
Let $B$ be a uniform elastic solid material; each undistorted
material $G$-structure $G(M)$ defines a Riemannian metric $g$,
invariant under material symmetries and isomorphisms.
\end{thm}

\begin{proof}
Given a cross-section $(U,\sigma)$ of a fixed undistorted material
$G$-structure  $G(B)$, let $X\in U$ and define
\begin{equation} \label{eq.uniform.elastic.solid.metric}
g_X^\sigma(v,w) := \brqt{\sigma(X)^{-1}\cdot v,\sigma(X)^{-1}\cdot
w},\quad \forall X\in U,\forall v,w\in T_XB,
\end{equation}
where $\brqt{\,,}$ is the Euclidean scalar product. Thus,
$g^\sigma$ is clearly a smoooth positive definite symmetric
bilinear tensor field on $U$, since it is nothing more than the
pullback of the Euclidean metric. Let us check that, in this
manner, the metric $g^\sigma$ does not depend on the chosen
cross-section $(U,\sigma)$. Given any other cross-section
$(V,\tau)$, let $X\in B$ be in the intersection of their domains
(if not empty, of course), then
\begin{equation} \label{eq.uniform.elastic.solid.metric.compatibility}
\begin{array}{rcl}
g_X^\sigma(v,w) &=& \brqt{\sigma(X)^{-1}\cdot v,\sigma(X)^{-1}\cdot w}\\
                &=& \brqt{Q\cdot\tau(X)^{-1}\cdot v,Q\cdot\tau(X)^{-1}\cdot w}\\
                &=& \brqt{\tau(X)^{-1}\cdot v,\tau(X)^{-1}\cdot w}\\
                &=& g_X^\tau(v,w),
\end{array}
\end{equation}
where we used the fact that, by hypothesis,
$Q=\sigma(X)^{-1}\cdot\tau(X)\in G$ is orthogonal.
\par Now, let $P\in\GG_{X,Y}(B)$ be a material isomorphism; there will
exist cross-sections $(U,\sigma),(V,\tau)$ such that
$P=\tau(Y)\cdot\sigma(X)^{-1}$. Then, we have
\begin{equation} \label{eq.uniform.elastic.solid.metric.invariance}
\begin{array}{rcl}
g_Y(P\cdot v,P\cdot w) &=& \brqt{\tau(Y)^{-1}\cdot P\cdot
v,\tau(Y)^{-1}\cdot P\cdot w}\\
                       &=& \brqt{\sigma(X)^{-1}\cdot v,\sigma(X)^{-1}\cdot w}\\
                       &=& g_Y(v,w).
\end{array}
\end{equation}
The metric we where looking for is just the metric $g$ defined in
\eqref{eq.uniform.elastic.solid.metric}.
\end{proof}

If we consider the orthogonal groupoid $\OO(B)$ related to this
metric, we have that the material groupoid is included in it,
$\GG(B)\subset\OO(B)$. Reciprocally, if $B$ is a smoothly uniform
material such that it can be endowed with a Riemannian metric for
which the material symmetries and isomorphisms are orthogonal
transformations, $\GG(B)\subset\OO(B)$, then $B$ must be an
elastic solid. Thus, elastic solids are completely characterized
by Riemannian metrics with the property of being invariant under
material symmetries and isomorphisms.

\begin{rem} \label{rmk.classical.undistorted.relations}
Given two material $G$-structures, $G_1(B)$ and $G_2(B)$, of a
uniform elastic solid $B$, we know that they must be related by
the right action of a linear isomorphism $F\in\Gl(3)$, that is
$G_2(B)=G_1(B)\cdot F$. Thus, if $G_1(B)$ is undistorted, the
$G$-structure $G_2(B)$ will be undistorted if and only if the
symmetric part $V$ of the left polar decomposition of $F$,
$F=V\cdot R$, lies in the centralizer of $G_1$, that is $V\in
C(G_1)$ (\cf\ \cite{Wang67}, proposition 11.3). But this does not
imply that $G_1(B)$ and $G_2(B)$ define the same metric, which is
true only if $V=I$.
\end{rem}

\subsection{Uniform Elastic Fluids} \label{sec.uniform.elastic.fluids}

There are similar results for fluids as for solids. In this case,
the fluid structure induces volume forms.

\begin{prop} \label{th.uniform.elastic.fluid.volume.form}
Let $B$ be a uniform fluid material, then each undistorted
material $G$-structure $G(B)$ defines a volume form $\rho$
invariant under material symmetries and isomorphisms.
\end{prop}

\begin{proof}
Given a cross-section $(U,\sigma)$ of a fixed undistorted material
$G$-structure $G(B)$, let us define on $U$ the volume form
\begin{equation} \label{eq.uniform.elastic.fluid.volume.form}
\rho_\sigma = \sigma^{*1}\wedge\sigma^{*2}\wedge\sigma^{*3},
\end{equation}
where $\sigma^*$ denotes the co-frame cross-section of $\sigma$,
that is $\sigma^*:U\To\FF^*B$ such that
$\sigma^{*i}(\sigma_j)\equiv\delta^i_j$ on $U$. Let us show that
the volume form $\rho_\sigma$ does not depend on the chosen
cross-section $(U,\sigma)$. In fact, let $(U,\sigma),(V,\tau)$ be
two cross-sections with non-empty domain intersection, then for
any $n$ vectors $v_1,\dots,v_n\in T_XB$, with $X\in U\cap V$, we
have
\begin{eqnarray*}
\rho_\sigma(v_1,\dots,v_n) &=& \det(v_i^j)\\
                           &=&
\det((\sigma^{-1}\tau)_i^k)\cdot\det(\tilde v_k^j)\\
                           &=& \rho_\tau(v_1,\dots,v_n),
\end{eqnarray*}
where we have used $v_i=v_i^j\sigma_i=\tilde v_i^j\tau_i$,
$v_i^j=(\sigma^{-1}\tau)_i^k\cdot\tilde v_k^j$ and
$\sigma^{-1}\tau\in U(n)$. Since the tangent vectors
$v_1,\dots,v_n$ are arbitrary, $\rho_\sigma$ and $\rho_\tau$
coincide on the intersection of their domains, $U\cap V$. Thus,
the volume form given in
\eqref{eq.uniform.elastic.fluid.volume.form} defines locally a
volume form $\rho$ on the whole material body $B$.
\par Let us see how $\rho$ is invariant under material symmetries and
isomorphisms. Given $P\in\GG_{X,Y}(B)$, there must exist
cross-sections $(U,\sigma),(V,\tau)$ such that
$P=\tau(Y)\cdot\sigma(X)^{-1}$. Then, we have
\begin{equation} \label{eq.uniform.elastic.fluid.volume.form.invariance}
\rho\circ P =
(P^{-1}\tau)^{*1}\wedge(P^{-1}\tau)^{*2}\wedge(P^{-1}\tau)^{*3} =
\sigma^{*1}\wedge\sigma^{*2}\wedge\sigma^{*3} = \rho,
\end{equation}
which finishes the proof.
\end{proof}

Considering now the induced unimodular groupoid $\UU(B)$, by the
invariance we have the inclusion $\GG(B)\subset\UU(B)$ which also
characterizes elastic fluids.

\section{Unisymmetry and Homosymmetry} \label{sec.unisymmetry}

As we have seen, the concept of homogeneity must be understood
within the framework of uniformity. But, there are materials that
are not uniform by their very definition, the so called
\emph{functionally graded materials}, or FGM for short. This type
of material can be made by techniques that accomplish a gradual
variation of material properties from point to point: for
instance, ceramic-metal composites, used in aeronautics, consist
of a plate made of ceramic on one side that continuously change to
some metal at the opposite face. The material properties are also
given through a constitutive equation like
\eqref{eq.constitutive.equation}. Therefore, we will have a notion
of material symmetry and the symmetry groups will be non-empty as
in the case of uniform materials. For a FGM material, the symmetry
groups at two different points are still conjugate, accordingly to
the following definition.

\begin{defn} \label{def.unisymmetric.isomorphism}
Given a functionally graded material $B$, let be $X,Y\in B$; we
say that a linear map $A:T_XB\To T_YB$ is a \emph{unisymmetric
(material) isomorphism} if it conjugates the symmetry groups of
$X$ and $Y$, namely,
\begin{equation} \label{eq.symmetry.group.conjugation.2}
\GG(Y) = A\cdot\GG(X)\cdot A^{-1}.
\end{equation}
\end{defn}

As for uniform bodies, the material properties of a FGM are now
characterized by the collection of all the possible unisymmetric
isomorphisms.

\begin{defn} \label{def.material.normalizoid}
Given a functionally graded material $B$, the set of unisymmetric
isomorphisms, that is the set
\begin{equation} \label{eq.material.normalizoid}
\NN(B) = \set{A\in\Pi(B)\ :\ \GG(Y) = A\cdot\GG(X)\cdot A^{-1} },
\end{equation}
will be called the \emph{FGM material groupoid} of $B$.
\end{defn}

\begin{figure}[h]
\centering
\includegraphics[scale=0.6]{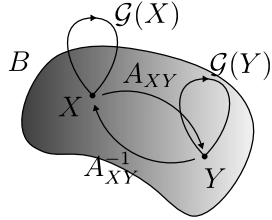}
\caption{The FGM material groupoid.}
\label{fig.material.normalizoid}
\end{figure}

We may now extend the ideas of section \secref{sec.uniformity}
using this new object. Then we obtain:

\begin{defn} \label{def.material.unisymmetry}
A functionally graded material $B$ will be said
\emph{unisymmetric} if the FGM material groupoid $\NN(B)$ is
transitive and, \emph{smoothly unisymmetric} if it is a Lie
groupoid.
\end{defn}

Note that the notion of unisymmetry covers a qualitative aspect in
the sense that a unisymmetric FGM is made of only one ``type'' of
material. For instance, it will be a fully isotropic solid
everywhere or a fluid everywhere, but it cannot be a fully
iscotropic solid at some point and a fluid at another point.

For this groupoid, we also have the associated $G$-structures.

\begin{defn} \label{def.material.nstructure}
Let $B$ be a smoothly unisymmetric body. Any of the asociated
$G$-strutures $\NN_z(B)$, with $z\in\FB$, will be called a
\emph{material $N$-structure}. A cross-section of a material
$N$-structure will be a \emph{unisymmetric cross-section} and a
configuration inducing such a cross-section will be a
\emph{unisymmetric configuration}. If for any of the material
$N$-structures there exists a covering by unisymmetric
configurations, the body $B$ will be said \emph{locally
homosymmetric}, and \emph{(globally) homosymmetric} if the
covering consists of only one unisymmetric configuration.
\end{defn}

As we may see, the homosymmetry property is equivalent to the
integrability of any of the material $N$-structures. However,
there is not an analogue result to Theorem
\ref{th.homogeneous.constitutive.equation.coordinates} for
homosymmetric bodies. Since, even if we have an $N$-structure and
the group structure is the same for any point through any
unisymmetric configuration, the symmetry groups may be represented
by different subgroups of $N$ at each point.

\subsection{Functionally Graded Elastic Solids}
\label{sec.functionally.graded.elastic.solids}

\begin{defn} \label{def.elastic.solid.classification}
We will say that a functionally graded elastic material $B$ is a
\emph{functionally graded solid} if there is a Riemmanian metric
on $B$ invariant under material symmetries, that is every point is
solid. Furthermore, $B$ will be said
\begin{enumerate}
\item \emph{fully isotropic} if every point is fully isotropic;
\item \emph{transversely isotropic} if every point is transversely
isotropic; and \item \emph{triclinic} if every point is triclinic.
\end{enumerate}
The compatible metric is called a \emph{material metric}.
\end{defn}

We have not used the term ``intrinsic'' for the material metric,
since it does not arise from the material structure as for uniform
elastic solids (\cf\ Theorem
\ref{th.uniform.elastic.solid.metric}). The material metric is an
extra structures that ensures that the solid points are glued in a
solid way.

If $B$ is a FGM solid and we consider the orthonormal
cross-sections $(U,\sigma)$ of the $O(3)$-structure given by a
solid metric, then they must verify:
\begin{eqnarray}
&\sigma(X)^{-1}\cdot\GG(X)\cdot\sigma(X)\subseteq O(3) \quad
\forall
X\in U \quad \forall(U,\sigma),&\\
&\sigma(X)^{-1}\cdot\tau(X)\in O(3) \quad \forall X\in U\cap V
\quad \forall(U,\sigma),(V,\tau);&
\end{eqnarray}
where $\GG(X)$ is the material symmetry group of $B$ at $X$. In
fact, these two conditions are necessary and sufficient to define
a solid metric compatible with the material structure by means of
a family of cross-sections of $\FB$.

On the other hand, if we consider another $O(3)$-structure, giving
a second solid metric, the two structures are not a priori related
by the right action of a linear isomorphism $F\in\Gl(3)$. But if
they are, then the symmetric part of the polar decomposition of
$F$ must be spherical, a homothety. This can be interpreted as the
material being in both cases in the same state but the measures of
stress, or strain, are performed with different scales.

\begin{defn} \label{def.relaxable.solid}
A solid FGM $B$ will be said to be \emph{relaxable} if the
$O(3)$-structure given by some solid metric is integrable or,
equivalently, if the Riemannian curvature (with respect to this
metric) vanishes identically. We then say that the
$O(3)$-structure is \emph{relaxed}.
\end{defn}

\begin{defn}
We say that a body $B$ is \emph{homosymmetrically relaxable} if
$B$ is an unisymmetric solid material for which there exists a
covering $\Sigma$ of local configuration that are both,
unisymmetric and relaxed configurations.
\end{defn}

Let $B$ be a homosymmetrically relaxable elastic solid, then we
have these two structures, the unisymmetric and the orthogonal,
which are in certain manner interconnected. As $B$ is a solid,
intuitively we may perceive that only the orthogonal part of a
unisymmetric isomorphism must be important. In what follows, we
will explain this fact in more detail.

A direct consequence of the previous Lemma
\ref{th.orthogonal.reduced.normaloid} and Proposition
\ref{th.intersection.groupoid} is the following theorem, which
implies a result proved by Epstein and de Le\'on \cite{EpsLe00}.

\begin{thm} \label{th.deleon.epstein}
If $B$ is relaxable elastic solid that is also homosymmetric, we
have
\begin{equation} \label{eq.orthogonal.reduced.normalizoid}
\bar\NN(B) = \NN(B)\cap\OO(B),
\end{equation}
where $\bar\NN(B)$ consits in the orthogonal part of the
isomorphisms of $\NN(B)$. Therefore, if $\bar\NN_z(B)$ is a smooth
$\bar N_z$-structure, $B$ will be homosymmetrically relaxable if
and only if the \emph{reduced material groupoid} $\bar\NN_z(B)$ is
integrable (where $z\in\FB$ is fixed).
\end{thm}

Let $B$ a relaxable and homosymmetric elastic solid and let $g$ denote the compatible material metric
\begin{itemize}
\item If $B$ is \emph{fully isotropic}, which means the symmetry
group $\GG(X)$ of each point $X\in B$ is equal to the orthogonal
group $O(T_XB,g)$ itself,
then the reduced FGM material groupoid $\bar\NN(B)$ coincides with
the orthogonal grupoid $\OO(B)$. \item If $B$ is \emph{triclinic}
(the only element of the symmetry group is the identity map), the
FGM material groupoid $\NN(B)$ is the full frame groupoid
$\Pi(B)$, and thus $\bar\NN(B)=\OO(B)$ as before. \item If $B$ is
\emph{transversally isotropic}, at each point $X\in B$ there
exists a basis of $T_XB$ in which the material symmetries
$g\in\GG(X)$ may be represented by matrices of the form:
$$\prth{\begin{matrix}1 & 0 & 0\\ 0 & \cos\theta & -\sin\theta\\ 0 &
\sin\theta & \cos\theta\end{matrix}}$$ Thus, for this basis, the
normalizer of $\GG(X)$ is
$$ \NN(X) = \brqt{\prth{\begin{matrix}1 & 0 & 0\\ 0 & \cos\theta &
-\sin\theta\\ 0 & \sin\theta &
\cos\theta\end{matrix}},\prth{\begin{matrix}\alpha & 0 & 0\\ 0 &
\beta & 0\\ 0 & 0 & \beta\end{matrix}}} $$ where the brackets
denote the group generated by the elements enclosed, and where
$\theta,\alpha,\beta$ are real numbers,  $\alpha,\beta$ being in
addition positive. Therefore, the group at any base point of the
reduced FGM material groupoid coincides with the respective
symmetry group, that is
$$ \bar\NN(X) = \GG(X) \quad \forall x\in B. $$
This means that, even if the material groupoid $\GG(B)$ (the set
consisting of material isomorphisms and symmetries) is not transitive
(\ie\ $B$ is not uniform), the reduced FGM material groupoid
$\bar\NN(B)$ is, and it coincides with $\GG(B)$ on the symmetry
groups. Thus, there is some kind of uniformity that generalizes
the classical one. Finally, note that any $G$-structure related to
$\bar\NN(B)$ will have a transversely isotropic structural group
as mentioned before.
\par Finally, note that we recover an analogue result to Theorem
\ref{th.homogeneous.constitutive.equation.coordinates}, which is
also true for fully isotropic FGM solids. If $B$ is
homosymmetrically relaxable, then for a unisymmetric and relaxable
configuration $K$, the constitutive equation will be invariant
under the action of the structure group of the reduced
$N$-stucture, related to the configuration $K$. In this case, the
structure group will coincide through $K$ with the symmetry group
$\GG_K(X)$ at any point $X$ in the domain of $K$. However, the
constitutive equation will not be independent of the point.
\end{itemize}

\subsection{Functionally Graded Elastic Fluids}
\label{sec.functionally.graded.elastic.fluids}

In the same way we have generalized the definition of elastic
solids in section \secref{sec.functionally.graded.elastic.solids},
we are going to give a new definition of elastic fluids.
Classically, an elastic fluid is a uniform elastic material which
posses a unimodular material structure, that is a $U(3)$-structure
(see \cite{TrueNoll65} for instance), even though there are
smaller fluid structures as the ones of fluid crystals (\cf\
\cite{LeMar04}).

\begin{defn} \label{def.functionally.graded.elastic.fluids}
We will say that a functionally graded elastic material $B$ is a
\emph{functionally graded fluid} (or a \emph{functionally graded
fluid crystal}) if there is a volume form $\rho$ on $B$ invariant
under material symmetries such that every point is fluid (or,
respectivelly, if every point is a fluid crystal). The volume form
is called a \emph{material form}.
\end{defn}

As in the case of functionally graded elastic solids, the
following two conditions on cross-sections $(U,\sigma)$ of the
frame bundle $\FB$,
\begin{eqnarray}
&\sigma(X)^{-1}\cdot\GG_x\cdot\sigma(X)\subseteq U(3) \quad
\forall
X\in U \quad \forall(U,\sigma)&\\
&\sigma(X)^{-1}\cdot\tau(X)\in U(3) \quad \forall X\in U\cap V
\quad \forall(U,\sigma),(V,\tau)&
\end{eqnarray}
characterize the fluid material structure.

Given a functionally graded elastic fluid $B$, consider the
unimodular groupoid $\UU(B)$ related to the volume form $\rho$
(Example \ref{eg.unimodular.groupoid}). When two fluid points have
conjugate symmetry groups, only the unimodular part of the
conjugate transformation plays a role in the conjugation. That is,
if $P$ is the transformation that conjugates these two groups,
then the unimodular transformation $P/{\det}_\rho(P)$ still
realizes the conjugation.

\begin{prop} \label{th.deleon.epstein.fluids}
If $B$ is a unisymmetric elastic fluid, then
\begin{equation} \label{eq.deleon.epstein.fluids}
\NN^1(B) = \NN(B)\cap\UU(B),
\end{equation}
where $\NN^1(B)$ is the unimodular reduction of the FGM material
groupoid.
\end{prop}

Let $B$ a fluid crystal of first kind (see
\cite{LeMar04,Wang67}), that is, an elastic fluid as in
\ref{def.functionally.graded.elastic.fluids} such that, for each
material point $X\in B$, the symmetry group $\GG(X)$ may be
represented for some reference $z$ at $X$ by matrices of the form
$$ A = \prth{\begin{matrix}a & b & 0\\ c & d & 0\\ e & f & g\end{matrix}} $$
with $\det(A)=\pm 1$. The normalizer in $\Gl(3)$ of this group of
matrices is the set of matrices of the same form but with the
restriction $\det(A)\neq0$. Therefore, when we intersect the
normalizer with $U(3)$ we obtain the original group of matrices.
This means that $\NN^1(X)=\GG(X)$ for every material point $x\in
B$.

The latter example shows us how a fluid material, which is not
necessarilly uniform, preserves uniformly the symmetry group
structure across the body.

\section*{Acknowledgements}
This work has been supported through a grant of the MEC,
Ministerio de Educaci\'on y Ciencia (Spain), project
MTM2007-62478. The third author aknowledges the MEC for an FPI
grant and the warm hospitality of the Department of Mechanical
Engineering, University of Calgary.


\begin{thebibliography}{10}
\bibitem{Bloom79}
F.~Bloom, \emph{Modern differential geometric techniques in the
theory of continuous distributions of dislocations}, Lecture Notes
in Math. \textbf{733}, Springer, Berlin, 1979.

\bibitem{ElEpsSni90}
M. El{\.z}anowski, M. Epstein, J. {\'S}niatycki,
\emph{{$G$}-structures and material homogeneity}, J. Elasticity
\textbf{23} (no.~2-3) (1990), 167--180.

\bibitem{EpsLe00}
M.~Epstein, M.~de~Le{\'o}n, \emph{Homogeneity without uniformity:
towards a mathematical theory of functionally graded materials},
Internat. J. Solids Structures \textbf{37} (no.~51) (2000),
7577--7591.

\bibitem{Fuji72}
A.~Fujimoto, \emph{Theory of {$G$}-structures}, Study Group of
Geometry, Department of Applied Ma\-the\-ma\-tics, College of
Liberal Arts and Science, Okayama University, Okayama, 1972.

\bibitem{KoNo69A}
S.~Kobayashi, K.~Nomizu, \emph{Foundations of differential
geometry. Vol.~I}, John Wiley \& Sons Inc., New York, 1996.

\bibitem{KoNo69B}
\bysame, \emph{Foundations of differential geometry. Vol.~II},
John Wiley \& Sons Inc., New York, 1996.

\bibitem{LeMar04}
M.~de~Le{\'o}n, D.~Mar{\'{\i}}n, \emph{Classification of material
{$G$}-structures}, Mediterr. J. Math. \textbf{1} (no.~4) (2004),
375--416.

\bibitem{Mack87}
K.~Mackenzie, \emph{Lie groupoids and {L}ie algebroids in
differential geometry}, London Mathematical Society Lecture Note
Series, vol.~124, Cambridge University Press, Cambridge, 1987.

\bibitem{MarsHu83}
J.~Marsden and T.~Hughes, \emph{Mathematical foundations of
elasticity}, Dover Publications Inc., New York, 1994.

\bibitem{Noll}
W. Noll, \emph{Materially uniform simple bodies with
inhomogeneities}, Arch. Rational Mech. Anal. \textbf{27}
(1967/1968), 1--32.

\bibitem{True77}
C.~Truesdell, \emph{A first course in rational continuum
mechanics. Vol.~1}, Academic Press, New York, 1977.

\bibitem{TrueNoll65}
C.~Truesdell and W.~Noll, \emph{The nonlinear field theories of
mechanics} (second ed.), Springer-Verlag, Berlin, 1992.

\bibitem{TrueWang73}
C.~Truesdell and C.~C.~Wang, \emph{Introduction to rational
elasticity}, Noordhoff International Publishing, Leyden, 1973.

\bibitem{Wang67}
C.-C.~Wang, \emph{On the geometric structures of simple bodies. A
mathematical foundation for the theory of continuous distributions
of dislocations}, Arch. Rational Mech. Anal. \textbf{27}
(1967/1968), 33--94.
\end{thebibliography}

\end{document}